\documentclass{article}

\usepackage{latexsym,mathrsfs,bm}
\usepackage{graphicx}
\usepackage{subcaption}
\usepackage[svgnames]{xcolor}
\usepackage{mathtools}
\usepackage{xifthen}

\usepackage[english]{babel}
\usepackage{paralist}
\usepackage{amsthm,amsfonts,amssymb,amsmath}
\usepackage[latin1]{inputenc}
\newcommand{\eps}{\varepsilon}
\newcommand{\R}{\mathbb{R}}
\newcommand{\Q}{\mathbb{Q}}

\newcommand{\N}{\mathbb{N}}

\newcommand{\Z}{\mathbb{Z}}

\newcommand{\es}[1]{\begin{equation}\begin{split}#1\end{split}\end{equation}}
\newcommand{\est}[1]{\begin{equation*}\begin{split}#1\end{split}\end{equation*}}

\renewcommand{\mod}[1]{~\pr{\textnormal{mod}~#1}}
\newtheorem*{theo*}{Theorem}
\newtheorem{theo}{Theorem}

\newtheorem{lemma}{Lemma}
\newtheorem{corol}[lemma]{Corollary}

\newtheorem*{rem*}{Remark}

\newcommand{\pr}[1]{\left( #1\right)}
\newcommand{\pg}[1]{\left\{ #1\right\}}
\newcommand{\pmd}[1]{\left| #1\right|}

\let\originalleft\left
\let\originalright\right
\renewcommand{\left}{\mathopen{}\mathclose\bgroup\originalleft}
\renewcommand{\right}{\aftergroup\egroup\originalright}
\newcommand{\comment}[1]{}
\makeatletter
\newcommand{\subjclass}[2][2010]{%
  \let\@oldtitle\@title%
  \gdef\@title{\@oldtitle\footnotetext{#1 \emph{Mathematics subject classification.} #2}}%
}
\newcommand{\keywords}[1]{%
  \let\@@oldtitle\@title%
  \gdef\@title{\@@oldtitle\footnotetext{\emph{Key words and phrases.} #1.}}%
}
\makeatother
\newcommand{\addresses}{{
  \bigskip
  \footnotesize

  S.~Bettin, \textsc{Dipartimento di Matematica, Universit\`a di Genova; via Dodecaneso 35; 16146 Genova; Italy. 
}\par\nopagebreak
  \textit{E-mail address}: \texttt{bettin@dima.unige.it}

 }}

\numberwithin{equation}{section}
\begin{document}

\title{A congruence sum and rational approximations}
\author {Sandro Bettin}


\maketitle

\begin{abstract}
We give a reciprocity formula for a two-variable sum where the variables satisfy a linear congruence condition. We also prove that such sum is a measure of how well a rational is approximable from below and show that the reciprocity formula is a simple consequence of this fact.
\end{abstract}

\section{A congruence sum}
Sums over the integers subject to some congruence conditions are ubiquitous in number theory and often appear in several other areas of mathematics. They are related to many classical arithmetic problems (e.g. that of primes in arithmetic progression, the Dirichlet's divisor problem, etc.) and even when they seem particularly innocuous there is actually a lot of arithmetic information hidden in them. This is the case of the sum
\est{
S\pr{a/q}
:=\sum_{\substack{am\equiv n\mod q,\\mn< q,\ n,m>0}}1,
}
where  $a,q\in\Z$, $q>0$, and $(a,q)=1$, which counts the number of points $(m,n)$ in $(\Z/q\Z)^2$ which belong to the line $am\equiv n\mod q$ and are contained in the hyperbolic region $mn< q$ (here of course we mean the representatives with $0<m,n\leq q$). Notice that $S$ can be interpreted as a $1$-periodic function defined over the rational numbers.

Two ``visual'' examples of such function are given in Figure~1 below.
\begin{figure}[!!ht]
\centering 
\includegraphics[scale=0.45]{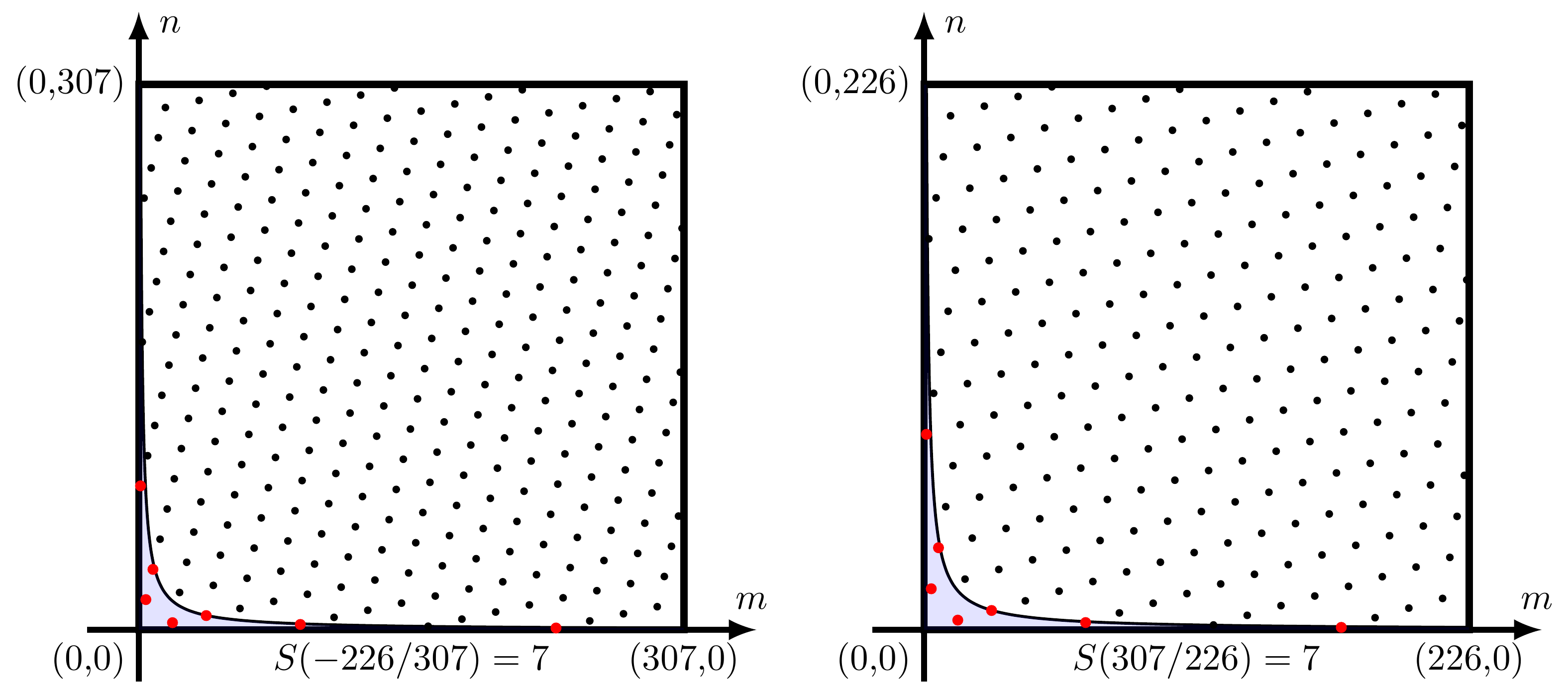}
\caption{A visual computation of $S(a/q)$ in the case $(a,q)=(-226,307)$ and $(a,q)=(307,226)$.
The points represent the couples $(m,n)\in\{0<r<q\}^2$ satisfying $am\equiv n\mod q$, the red ones being the ones below the hyperbola $mn=q$.}
\end{figure}
Notice that in both cases $S(-226,307)=S(307,226)=7$ and the two graphs look overall very similar. This is actually not a coincidence: indeed one always have that $S(-a,q)$ is close to $S(q,a)$ (for $q=307$ their difference is always $\leq3$ and is typically either $0$ or $1$, whereas the maximum value of $S(a,307)$ is $17$). A similar relation holds for $S(a,q)$ and $S(-q,a)$, with the difference that in this case one also has a main term. These two facts are expressed in the following theorem.
\begin{theo}\label{fh}
Let $1\leq a<q$. Then,
\begin{align}
S(a/q)-S(-q/a)&=\sqrt{q/a}+E_{+}(a,q),\label{fp}\\
S(-a/q)-S(q/a)&=E_{-}(a,q),\label{sp}
\end{align}
with $|E_{\pm}(a,q)|\leq \frac32 k_{a/q}+3$ where $k_{a/q}$ is the number of steps in the Euclid division $q/a$ (in particular $E_{\pm}(a,q)\ll \log(2+q)$).
\end{theo}

The $\leq$ part of these inequalities was already obtained by Young~\cite{You} in his extension of Conrey's reciprocity formula for the twisted second moment of Dirichlet $L$-functions~\cite{Con}. Young also obtained a similar version of these formulas in the case when the sharp cut-off $mn\leq q$ is replaced by a smooth one, that is by inserting a factor of $f(mn/q)$ where $f(x)$ is a smooth function going to zero faster than any polynomial at $+\infty$ (see also~\cite{Bet15} for an alternative treatment of the smoothed case). We remark that the same method used to prove Theorem 1 can be used to obtain a simpler proof of the Conrey-Young reciprocity formula for the twisted second moment.

Theorem~\ref{fh} suggests that one can obtain a formula for $S(a/q)$ in terms of the coefficients of the continued fraction expansion $[0;b_1,\cdots,b_k]$ of $a/q$. Indeed one can repeatedly alternate the use of one among~\eqref{fp} or~\eqref{sp} and the reduction modulo the denominator and obtain
\est{
S(a/q)=\sum_{j=0}^{[(k-1)/2]}\sqrt{b_{2j+1}}+O(k_{a/q}^2).
}
However, in fact one can prove directly a stronger form of this result, with $O(k_{a/q}^2)$ replaced by $O(k_{a/q})$, and then deduce Theorem~\ref{fh} from it. 
\begin{theo}\label{mct}
Let $a,q\in\Z$ with $q>0$, $a\neq 0$ and $(a,q)=1$. Let $[0;b_1,\cdots,b_k]$ be the continued fraction expansion of $\frac aq$. Then
\est{
S(a/q)&=\sum_{j=0}^{[(k-1)/2]}[\sqrt{b_{2j+1}}]+E,
}
where $0\leq E\leq \frac32 k+1\ll \log(2+ q)$ and $[x]$ denotes the integer part of $x$.
\end{theo}

Theorem~\ref{mct} doesn't provide a useful asymptotic formula for most values of $a$. Indeed, a standard (but somewhat lengthy) computation with Dirichlet's hyperbola method shows that the average value of $S(a/q)$ is
\est{
\frac1{\varphi(q)}\sum_{\substack{0<a<q,\\ (a,q)=1}}S(a/q)&=\frac{\varphi(q)\sigma_{-1}(q)}{q}\pr{\log q+O\pr{\log\log q}}\asymp \log q,\\
}
as $q\rightarrow\infty$, where $\sigma_{-1}(q)=\sum_{d|q}d^{-1}$ and $\varphi$ is Euler's totient function.\footnote{Notice that if $q$ is prime, then the average of $S(a/q)$ reduces exactly to the Dirichlet's divisor problem: $\frac1{\varphi(q)}\sum_{\substack{0<a<q,\\ (a,q)=1}}S(a/q)=\sum_{n<q}d(n)$, where $d(n)$ is the number of divisors of $n$.}
However, Theorem~\ref{mct} is still useful as it determines exactly all the large values of $S(a/q)$. Indeed, the Theorem shows that $S(a/q)$ is ``large'' if and only if the continued fraction expansion of $a/q$ has a ``large'' odd coefficient. 

Theorem~\ref{mct} can also be used to prove some density results for $S(a/q)$, following the method introduced by Hickerson~\cite{Hic} for the Dedekind sum. For example, we can prove the following.
\begin{corol}\label{2c}
Let $\kappa>1$. Then the set $\pg{\pr{\frac aq,\frac{S(a/q)}{\log^\kappa(2+q)} }\mid \frac aq\in \Q}$ is dense in $\R\times \R_{>0}$.
\end{corol}
It would be interesting to understand how the set $\{(\frac aq,S(a/q))\mid \frac aq\in \Q\}$ is distributed as $q\rightarrow\infty$ (see~\cite{Var} and~\cite{Bet15} for the computation of the distribution of two somewhat similar sums via the use of reciprocity formulas close to~\eqref{fp} and~\eqref{sp}).

\section{Rational approximations}
Theorem~\ref{fh} and~\ref{mct} actually admit a rather simple explanation, once one realizes that $S(a/q)$ can be defined also in another simple and apparently unrelated way. Indeed, $S(a/q)$ coincides with the number of ways $a/q$ can be ``well approximated'' from below by fractions of smaller denominator. The theory of continued fractions tells us that ``good approximations'' of a real number are obtained by its convergents, which then explains why the coefficients of the continued fraction expansion arise in Theorem~\ref{mct}.

Before stating the precise result, we give an example for the case of $S(-226/307)$ considered in Figure~1. The $7$ points $(m,n)$ appearing in the sum defining $S(-226/307)$ are $(235,1)$, $(91,3)$, $(19,4)$, $(38,8)$, $(4,17)$, $(8,34)$, $(1,81)$. Now, compare them with the fractions $r/s$ with $s<307$ which satisfy $0<-\frac{226}{307}-\frac rs<\frac{1}{s^2}$. These are the even convergents $-\frac11,-\frac{3}{4}$, $-\frac{14}{19}$, the ``double'' of the last two $-\frac{6}{8}$, $-\frac{28}{38}$, and two of the semi-convergents $-\frac{67}{91}$ and $-\frac{173}{235}$. Notice that the denominators of such fractions are precisely the $m$-coordinates of the $7$ points! The following simple theorem tells us that this is no coincidence.
\begin{theo}\label{revat}
Let $a,q\in\Z$, with $q>0$ and $(a,q)=1$. Then,
\es{\label{reva}
S(a/q)=\#\{(c,d)\in\Z^2\mid d>0,\ 0<\tfrac{a}{q}-\tfrac cd< \tfrac 1{d^2}\}.
}
\end{theo}

\begin{proof}[Proof of Theorem~\ref{revat}]
Let $0<n,m<q$ with $am\equiv n\mod q$. Let $r$ be the largest integer such that $\frac{r}{m}<\frac aq$; clearly $0<\frac aq-\frac{r}{m}<\frac 1m$. Also, 
\es{\label{ade}
am=q\frac{a}qm= q\pr{\frac{r}{m}+\frac aq-\frac{r}{m}}m\equiv q\pr{\frac aq-\frac{r}{m}}m\mod{q}.
}
Thus, since $0<q(\frac aq-\frac{r}{m})m<q$, we must have $q(\frac aq-\frac{r}{m})m=n$. We then have that $mn< q$ if and only if $q(\frac aq-\frac{r}{m}) m^2< q$ and thus if and only if $(\frac aq-\frac{r}{m})< \frac 1{m^2}$. Thus, we obtain the $\leq$ side of~\eqref{ade}; it is clear that repeating the same argument in the opposite direction yields the other inequality and so the proof is completed.
\end{proof}

Notice that the set on the right hand side of~\eqref{reva} doesn't have the condition $(c,d)=1$, that is we are not requiring the fractions $\frac cd$ to be in reduced form. If one wants to count only reduced fractions, then one immediately sees that $S(a/q)$ corresponds to a weighted sum, where the weight takes into account how good the approximation is:
\es{\label{fasd}
S(a/q)=\sum_{\substack{(c,d)=1,\, d>0\\0<\frac{a}{q}-\frac cd\leq \frac 1{d^2}}}[\eps_{c/d}^{-\frac12}],
}
where $\eps_{c/d}=d^2|{\frac{a}{q}-\frac cd}|$. The theory of continued fractions then helps us recover Theorem~\ref{mct} from~\eqref{fasd} and Theorem~\ref{fh} will then follow easily.

Before proving the various results we make one last remark. If in $S(a/q)$ we add the condition $(m,n)=1$ and include also the solutions of $am\equiv -n\mod q$, then we have the following analogue of Theorem~\ref{revat}:
\es{\label{ard}
\sum_{\substack{am\equiv \pm n\mod q,\\mn< q,\ n,m>0,\\(m,n)=1}}1=\#\{\tfrac cd\in\Q\mid,\ (c,d)=1,\  0\neq \pmd{\tfrac{a}{q}-\tfrac cd}< \tfrac 1{d^2}\}
}
valid for $q$ prime (if $q$ is not prime one has to add the condition $(q,d)=1$ in the set on the right hand side).
Now, using M\"obius inversion formula and the asymptotic for Dirichlet's divisor problem, we have that for $q$ prime the average value of the left hand side is
\est{
\frac1{\varphi(q)}\sum_{\substack{a\mod q,\\(a,q)=1}}\sum_{\substack{am\equiv \pm n\mod q,\\mn< q,\ n,m>0,\\(m,n)=1}}1
=\frac{12}{\pi^2}\log q+O(1),
}
as $q\rightarrow\infty$. Also, it is known (cf. the next Section) that all the convergents to $\frac aq$ (different from $\frac aq$) are contained in the set of the right hand side of~\eqref{ard} and that, on average over $a$, there are asymptotically $\frac{12\log 2}{\pi^2}\log q$ convergents of $\frac aq$ as $q$ goes to infinity among primes~\cite{Hei}. In particular, on average over $a$ and as $q\rightarrow\infty$ among primes, we have that $\log 2\approx 69.3\%$ of the solutions to $|\frac{a}{q}-\frac cd|< \frac 1{d^2}$ are partial quotients of $\frac aq$, whereas $\approx 30.7\%$ are not. (In Lemma~\ref{mctl} below we will see that these  other solutions are certain semi-convergents of $a/q$).

We conclude this section by observing that the second moment of $S(a/q)$ is very closely related to the $4$-th moment of Dirchtlet L-functions at the central point. In particular, Theorem 3 opens an alternative approach to this problem via methods of Diophantine approximation (see also~\cite{CK}).

\section{Proofs of Theorems~\ref{mct} and~\ref{fh} and of Corollary~\ref{2c}}

\begin{proof}[Proof of Theorem~\ref{mct}]
It is well known (see~\cite{Khi}, Chapter~1
) that all convergents $h_{j}/k_j$ of $a/q$ (with $\frac {h_j}{k_j}\neq \frac aq$) satisfy
\est{
\frac{1}{k_j(k_{j+1}+k_j)}<\pmd{\frac{h_{j}}{k_j}-\frac aq}<\frac{1}{k_jk_{j+1}}<\frac1{k_{j}^2}
}
and that one has $\frac{h_{j}}{k_j}<\frac aq$ if and only if $j$ is even, so that all the even convergents appear in the sum~\eqref{fasd}. Also, the above inequalities give
\est{
\frac{k_j}{(k_{j+1}+k_j)}<\eps_{h_{j}/k_j}<\frac{k_j}{k_{j+1}}
}
and so, since $k_{j+1}=b_{j+1}k_{j}+k_{j-1}$ and $0\leq k_{j-1}< k_j$ for $j\geq0$, we obtain $b_{j+1}\leq \eps_{h_{j}/k_j}^{-1}\leq b_{j+1}+2$. By the inequality $[\sqrt {x+2}]-[\sqrt x]\leq1$ for $x\geq1$ we then have $[b_{j+1}^\frac12]\leq [\eps_{h_{j}/k_j}^{-\frac12}]\leq [b_{j+1}^\frac12]+1$.

The solutions $c/d$ to $0<|\frac{c}{d}-\frac aq|<\frac1{d^2}$ are not all convergents of $a/q$, however this is not far for being true. Indeed, all solutions $c/d$ with $(c,d)=1$ of the stricter inequality $0<|\frac{c}{d}-\frac aq|<\frac1{2d^2}$ are convergents (see~\cite{Khi}, Theorem~19), so that $|\eps_{c/d}|^{-1}\leq2$ if $c/d$ is not a convergent. In particular,~\eqref{fasd} gives
\est{\label{fasd2}
S(a/q)=\sum_{\substack{j=1,\\ j\text{ odd}}}^k[b_{j}^{\frac12}]+S^*(a/q)+\mathcal E_{a,q}
}
where $0\leq \mathcal E_{a,q}\leq (k+1)/2$ and $S^*(a/q)$ is the number of reduced rationals $c/d$ satisfying $0<\frac{a}{q}-\frac cd<\frac 1{d^2}$ which are not convergents of $a/q$. By Lemma~\ref{mctl} below $S^*(a/q)$ is bounded by twice the number of even convergents of $\frac aq$ (different from $a/q$), so that $S^*(a/q)\leq k$ and the proof of Theorem~\ref{mct} is completed.\end{proof}

\begin{lemma}\label{mctl}
Let $\frac aq\in\Q$ with continued fraction expansion $[b_0;b_1,\dots,b_k]$ and convergents $\frac{h_j}{k_j}$ for $-1\leq j\leq k$. Then, every $\frac cd\in\Q$ which satisfies $0\neq |\frac aq-\frac cd|<\frac1{d^2}$ with $(c,d)=1$ is either a convergent $\frac{h_j}{k_j}$ of $\frac aq$ (with $-1\leq j\leq k-1$) or a semi-convergent of the form $\frac{h_{j}+gh_{j+1}}{k_{j}+gk_{j+1}}$ with $g=1$ or $g=b_{j+2}-1$ (and $-1\leq j\leq k-2$). Moreover, in both cases $c/d< a/q$ if and only if $j$ is even. 
\end{lemma}
\begin{proof}
We assume $\frac aq<\frac cd$, the proof being essentially identical otherwise.
Now, if $0<\frac aq-\frac cd<\frac1{d^2}$ with $(c,d)=1$, then $\frac cd$ is a ``best rational approximation from below'' for $a/q$, that is $\frac cd$ is closer to $\frac aq$ than any rational which is less than $a/q$ and have a smaller denominator. Indeed, if $0<d'<d$ and $0<\frac{a}{q}-\frac {c'}{d'}\leq \frac{a}{q}-\frac {c}{d}$ for some $c'\in\Z$, then
\est{
\frac{1}{dd'}\leq \frac{c'}{d'}-\frac {c}{d}\leq \pr{\frac{a}{q}-\frac {c}{d}}-\pr{\frac{a}{q}-\frac {c'}{d'}}<\frac 1{d^2}
}
which gives a contradiction. Now,  all the best rational approximation from above or below for $\frac aq$ which are not convergents are semi-convergents of the form $\frac{h_{j}+gh_{j+1}}{k_{j}+gk_{j+1}}$ for $g$ an integer satisfying $1\leq g<b_{j+2}$ (Theorem~15 of~\cite{Khi} proves this for best rational approximations, but the proof carries over also for best approximations from below or from above); also this fraction is smaller than $\frac aq$ if and only if $j$ is even. We will now show that among these only the values $g=1$ and $g=b_{j+2}-1$ might be such that the inequality $|\frac{h_{j}+gh_{j+1}}{k_{j}+gk_{j+1}}-\frac aq|<\frac 1{(k_{j}+gk_{j+1})^2}$ is satisfied.

Clearly we can assume $b_{j+2}\geq4$, since otherwise there is nothing to prove.
For $j$ even we have $k_{j+1}h_{j}-k_{j}h_{j+1}=1$ and so
\est{
\frac{h_{j+2}}{k_{j+2}}-\frac{h_{j}+gh_{j+1}}{k_{j}+gk_{j+1}}=\frac{b_{j+2}-g}{(k_{j}+b_{j+2}k_{j+1})(k_{j}+gk_{j+1})}.
}
Since $\frac{h_{j}+gh_{j+1}}{k_{j}+gk_{j+1}}<\frac{h_{j+2}}{k_{j+2}}\leq \frac aq$, then it follows that in order to have $\frac aq-\frac{h_{j}+gh_{j+1}}{k_{j}+gk_{j+1}}<\frac1{(k_{j}+gk_{j+1})^2}$ it is necessary that 
\est{
(b_{j+2}-g)(k_{j}+gk_{j+1})<(k_{j}+b_{j+2}k_{j+1}).
}
and a fortiori one must have $(b_{j+2}-g)g<b_{j+2}$, since $b_{j+2}-g\geq1$. Solving for $g$ we obtain 
\est{
g<\frac{b_{j+2}-\sqrt{b_{j+2}^2-4b_{j+2}}}{2}\qquad \text{or}\qquad g>\frac{b_{j+2}+\sqrt{b_{j+2}^2-4b_{j+2}}}{2}
}
and thus, reminding that $g$ is a positive integer greater than or equal to $2$, the only possibilities are $g=1$ and $g=b_{j+2}-1$, as desired.
\end{proof}

\begin{proof}[Proof of Theorem~\ref{fh}]
A simple computation shows that if $a/q=[b_0;b_1,b_{2},\cdots b_{k}]$ with $b_0\geq0$ and $a\neq0$, then the continued fraction expansion of $-\frac qa$ is
\es{\label{fpp}
-\frac qa=\begin{cases}
      [-b_1-1;1,b_2-1,b_3,b_4,\dots,b_k]& \text{if $b_2=1$},\\
      [-b_1-1;b_3+1,b_4,b_5,\dots,b_k]& \text{if $b_2>1$},
     \end{cases}
}
if $b_0=0$ and $k\geq2$ (if $k=1$ then $-\frac{q}{a}=-b_1$)  and
\es{\label{spp}
-\frac qa=\begin{cases}
      [-1;1,b_0-1,b_1,b_2,\dots,b_k]& \text{if $b_0=1$},\\
      [-1;b_1+1,b_2,b_3\dots,b_k]& \text{if $b_0>1$},
     \end{cases}
}
if $b_0>0$ (if $k=0$ and $b_0=1$ then $-\frac qa=-1$). Equation~\eqref{fp} then follows from~\eqref{fpp} and Theorem~\ref{mct} upon observing that $b_1\leq q/a<b_1+1$ and thus $[\sqrt {b_1}]\leq \sqrt {q/a}<[\sqrt {b_1}]+1$. Equation~\eqref{sp} follows in the same way from~\eqref{spp} exchanging the roles of $a$ and $q$.
\end{proof}

\begin{proof}[Proof of Corollary~\ref{2c}]
Let $(x,y)\in\R\times \R_{>0}$. For every fixed $\eps>0$ we need to find $a/q$ such that $|a/q-x|<\eps$ and $|S(a/q)/\log^\kappa(2+q)-y|<\eps$. Since $S(a/q)$ is 1-periodic, we can assume $0<x<1$. Let $b_1,\cdots,b_{2r}$ be such that $|x-[0;b_1,\cdots b_{2r}]|\leq \eps/2$ for some $r\ll_\eps 1$. For $m,n\in \N$, let $\frac aq=[0;b_1,\cdots b_{2r},m,n]$ so that $q=(dm+d')n+d$ for some $d,d'\ll1$ (considering $b_1,\cdots b_{2r},x,y,r$ as fixed). Thus, by Theorem~\ref{mct}, we have
\est{
\frac{S(a/q)}{\log^\kappa(2+q)}=\frac{\sqrt{m}}{\log^\kappa(2+mn)}+o(1).
}
as $m,n\rightarrow\infty$. We take $n=[e^{m^\frac1{2\kappa}y^{-\frac1\kappa}}]$ and the Corollary follows by taking $m$ large enough.
\end{proof}

\section*{Acknowledgement}
This note was mostly written while the author was a postdoctoral fellow at the Centre de recherches math\'ematiques in Montr\'eal.

\appendix

\addresses

\begin{thebibliography}{0}

\bibitem[Bet15]{Bet15}
Bettin, S. \emph{On the distribution of a cotangent sum}. Int Math Res Notices (2015) 2015 (21): 11419--11432.

\bibitem[Bet16]{Bet16}
Bettin, S. \emph{On the reciprocity law for Dirichlet $L$-functions}. Trans. Amer. Math. Soc. 368 (2016), 6887--6914.

\bibitem[Con]{Con}
Conrey, J.B. \emph{The mean-square of Dirichlet L-functions}. arxiv math.NT/0708.2699.

\bibitem[CK]{CK}
Conrey, J.B.; Keating, J.P. \emph{Moments of Zeta and Correlations of Divisor-Sums: II}. Advances in the Theory of Numbers. Volume 77 of the series Fields Institute Communications pp. 75--85.

\bibitem[Hei]{Hei}
Heilbronn, H. \emph{On the average length of a class of finite continued fractions}.
1969 Number Theory and Analysis (Papers in Honor of Edmund Landau) pp. 87--96 Plenum, New York

\bibitem[Hic]{Hic}
Hickerson, D. \emph{Continued fractions and density results for Dedekind sums}. J. Reine Angew. Math. 290 (1977), 113--116.

\bibitem[Khi]{Khi}
Khinchin, A.Y. \emph{Continued fractions}. The University of Chicago Press, Chicago, Ill.-London 1964.

\bibitem[Var]{Var}
Vardi, I. \emph{Dedekind sums have a limiting distribution}. Int Math Res Notices (1993) 1993 (1): 1--12. 

\bibitem[You]{You}
Young, M.P. \emph{The reciprocity law for the twisted second moment of Dirichlet L-functions}. Forum Math. 23 (2011), no. 6, 1323--1337.

\end{thebibliography}
\end{document}